\documentclass{article}
\usepackage[utf8]{inputenc}

\usepackage[ruled,linesnumbered]{algorithm2e}
\usepackage{authblk}
\usepackage{amsmath}
\usepackage{amssymb}
\usepackage{amsthm}
\usepackage{bbm}
\usepackage{cite}
\usepackage[colorlinks]{hyperref}
\usepackage[nameinlink]{cleveref}
\usepackage{color}
\usepackage{enumitem}
\usepackage{mathtools}
\usepackage{caption}
\usepackage{floatrow}
\usepackage{subcaption}

\usepackage{tikz}
\usepackage{pgfplots}

\usepackage{booktabs}
\usepackage{multirow}

\usepackage{siunitx}

\theoremstyle{plain}
\newtheorem{theorem}{Theorem}
\newtheorem{lemma}[theorem]{Lemma}
\newtheorem{proposition}[theorem]{Proposition}
\newtheorem{corollary}[theorem]{Corollary}

\theoremstyle{definition}
\newtheorem{definition}[theorem]{Definition}

\theoremstyle{remark}

\crefname{theorem}{Theorem}{Theorems}
\Crefname{theorem}{Theorem}{Theorems}
\crefname{definition}{Definition}{Definitions}
\Crefname{definition}{Definition}{Definitions}
\crefname{remark}{Remark}{Remarks}
\Crefname{remark}{Remark}{Remarks}
\crefname{algorithm}{Algorithm}{Algorithms}
\Crefname{algorithm}{Algorithm}{Algorithms}
\crefname{proposition}{Proposition}{Propositions}
\Crefname{proposition}{Proposition}{Propositions}
\crefname{lemma}{Lemma}{Lemmas}
\Crefname{lemma}{Lemma}{Lemmas}
\crefname{corollary}{Corollary}{Corollaries}
\Crefname{corollary}{Corollary}{Corollaries}

\newcommand{\RA}{\mathcal{RA}}
\SetKwRepeat{Do}{do}{while}

\title{Non-uniform Grid Refinement for the Combinatorial Integral Approximation}

\newcommand{\dd}{\,\mathrm{d}}
\newcommand{\N}{\mathbb{N}}
\newcommand{\R}{\mathbb{R}}

\newcommand{\vol}{\operatorname{vol}}
\newcommand{\TT}{\mathcal{T}}
\newcommand{\Ha}{\mathcal{H}}
\newcommand{\weakstarto}{\stackrel{\ast}{\rightharpoonup}}

\newcommand{\Crefpart}[2]{%
  \hyperref[#2]{\namecref{#1}~\labelcref*{#1}.\ref*{#2}}%
}

\newcommand{\Crefparts}[3]{%
  \hyperref[#2]{\namecref{#1}~\labelcref*{#1}.\ref*{#2}--\labelcref*{#1}.\ref*{#3}}%
}

\definecolor{darkgreen}{rgb}{0,0.5,0}

\author[1]{Felix Bestehorn}
\author[2]{Christoph Hansknecht}
\author[1]{Christian Kirches}
\author[3]{Paul Manns}
\affil[1]{Technical University of Braunschweig}
\affil[2]{Technical University of Clausthal}
\affil[3]{TU Dortmund University}

\begin{document}
\maketitle

\begin{abstract}
The combinatorial integral approximation (CIA) is a solution technique
for integer optimal control problems. In order to regularize
the solutions produced by CIA, one can minimize switching costs in one
of its algorithmic steps. This leads to combinatorial optimization
problems, which are called switching cost aware rounding problems (SCARP). They can be
solved efficiently on one-dimensional domains but no efficient solution algorithms
have been found so far for multi-dimensional domains.

The CIA problem formulation depends on a discretization grid. We propose
to reduce the number of variables and thus improve the computational
tractability of SCARP by means of a non-uniform grid refinement strategy.
We prove that the grid refinement preserves the approximation properties
of the combinatorial integral approximation. Computational results 
are offered to show that the proposed approach is able to achieve, within a prescribed time limit, smaller duality gaps that does the uniform approach. For several large instances, a dual bound could only be obtained through adaptivity.
\end{abstract}

\section{Introduction}
We are concerned with a grid refinement technique in the context of mixed-integer
optimal control problems (MIOCPs). MIOCPs provide models for many different applications
from gas network optimization, see, e.g., \cite{martin2006mixed,hante2020mixed},
over traffic light control, see, e.g., \cite{goettlich2017partial},
and energy management of buildings, see, e.g., \cite{zavala2010proactive}, to
automatic control of automotive gear shifts, see, e.g., \cite{gerdts2005solving,kirches2010time}.

The combinatorial integral approximation
\cite{jung2014relaxations,kirches2021compactness}
is an approximation technique that has been designed
for MIOCPs that can be stated in the form
\begin{gather}\label{eq:p}
\inf_{\omega} j(S(\omega)) \text{ s.t.\ }
\left\{\begin{aligned}
&\omega \in L^\infty(\Omega, \R^M),\\
&\omega(s) \in \{0,1\}^M \text{ for almost every (a.e.) }
s \in \Omega,\text{ and}\\ 
&\sum_{m=1}^M \omega_m(s) = 1 \text{ for a.e.\ }
s \in \Omega,
\end{aligned}\right.
\tag{P}
\end{gather}
where $M \in \N$, $T > 0$, $\Omega \subset \R^d$, $d\in \N$,
is a bounded domain, and
$S : L^\infty(\Omega,\R^M) \to L^2(\Omega)$ 
maps weakly-$^*$ converging $L^\infty(\Omega,\R^M)$
to norm converging sequences in $L^2(\Omega)$. 
Moreover, we assume $j \in C(L^2(\Omega),\R)$. The function $j$
is the objective of the MIOCP and $S$ is a
\emph{control-to-state} operator of an underlying ordinary or partial
differential equation (ODE or PDE). The binary-valued
function $\omega$ is the control input of the system.
The constraint $\sum_{m=1}^M \omega_m(s) = 1$ a.e.\ in
$\Omega$ implies that exactly one entry of $\omega$
is set to one at a given point in the domain, while all others
are zero. The constraint may thus be interpreted as a 
\emph{one-hot encoding} of different modes of operation.

We present our theoretical findings in this setting and abstract from the
specific underlying ODE or other dynamical system. We note that ODEs
with Lipschitz continuous right-hand sides or elliptic PDEs
allow to verify the prerequisites on $S$
\cite{manns2021approximation,manns2020improved,manns2020multidimensional}.

The combinatorial integral approximation considers a continuous relaxation
of the problem formulation \eqref{eq:p}. Specifically, the \emph{one-hot encoding}
constraint is relaxed to convex coefficients. The resulting relaxed problem
reads
\begin{gather}\label{eq:r}
\min_{\alpha} j(S(\alpha)) \text{ s.t.\ }
\left\{\begin{aligned}
&\alpha \in L^\infty(\Omega, \R^M),\\
&\alpha(s) \in \{0,1\}^M \text{ for a.e.\ }
s \in \Omega,\text{ and}\\ 
&\sum_{m=1}^M \alpha_m(s) = 1 \text{ for a.e.\ }
s \in \Omega.
\end{aligned}\right.
\tag{R}
\end{gather}
Writing \emph{min} in \eqref{eq:r} and \emph{inf} in \eqref{eq:p} is deliberate because,
in contrast to \eqref{eq:p}, the problem \eqref{eq:r}
always admits a minimizer in our setting,
see, e.g., \cite{kirches2021compactness}.

The compactness assumption that $S$ maps weakly-$^*$ converging sequences
to norm converging sequences has important consequences. In particular,
\begin{gather}\label{eq:cia_identity}
\inf \left\{ j(S(\omega))\,|\, \omega \text{ is feasible for } \eqref{eq:p} \right\}
= \min \left\{ j(S(\alpha))\,|\, \alpha \text{ is feasible for } \eqref{eq:r} \right\},
\end{gather}
which can be considered as a transfer of the
Filippov--Waszewski theorem 
\cite{filippov1962certain,wazewski1963optimal}
and the Lyapunov convexity theorem
\cite{lyapunov1940completely,lindenstrauss1966short}
to MIOCPs \cite{kirches2021compactness}.

So-called \emph{rounding algorithms} allow to
exploit \eqref{eq:cia_identity} algorithmically.
They start from a feasible or optimal solution $\alpha^*$
of \eqref{eq:r}, which is, e.g., obtained with a gradient-based method for 
ODE- or PDE-constrained optimization problems,
and compute a sequence of control functions $(\omega^k)_k$. The  $\omega^k$ 
are feasible for \eqref{eq:p} and converge weakly-$^*$ to $\alpha^*$ in
$L^\infty(\Omega,\R^M)$. If $\alpha^*$ is a minimizer of \eqref{eq:r},
one obtains a minimizing sequence for \eqref{eq:p}
\cite{kirches2021compactness}.
Examples of rounding algorithms are sum-up rounding \cite{sager2005numerical,sager2006numerical,kirches2020approximation,sager2012integer,manns2020multidimensional,manns2021approximation,hante2013relaxation},
next-forced rounding \cite{jung2014relaxations},
(adaptive) maximum dwell rounding
\cite{sager2021mixed}, and switching cost aware rounding
\cite{bestehorn2019switching,bestehorn2020matching,bestehorn2021mixed}.

Rounding algorithms operate on some given $\alpha^*$ and discretizations of $\Omega$
indexed by $k \in \N$, to drive the sequence of error quantities
\begin{gather}\label{eq:error_measure}
d^k(\alpha^*,\omega^k) \coloneqq
\sup \left\{ \left\|\int_{\bigcup_{i=1}^n T_i^k}
\alpha^*(s) - \omega^k(s)\dd s\right\|_\infty \,\middle|\,
n \in \{1,\ldots,N^{k}\}
\right\}
\end{gather}
to zero for $k \to \infty$, where the functions $(d^k)_k$ are a family
of pseudometrics that are induced by a sequence of deliberately
refined grids that partition $\Omega$ into $T_1^k,\ldots,T_{N^k}^k$,
$N^k \in \N$, see
\cite[Definition 4.3]{manns2020multidimensional}.
The $\omega^k$ are cellwise-constant functions on the partition$\{ T_1^k,\ldots, T_{N^k}^k\}$
of $\Omega$. The rounding algorithms generally ensure $d^k(\alpha^*,\omega^k) \le C \vol^k$,
where $C \ge 1$ is a discretization-independent constant and $\vol^k$ is the \emph{grid constant}
defined as $\vol^k \coloneqq \max\left\{ \lambda(T_n^k) \,\middle|\, 
n \in \{1,\ldots,N^k\}\right\}$, where $\lambda$ denotes the Lebesgue
measure on $\R^d$.
The aforementioned refinement implies $\vol^k \to 0$
and thus $d^k(\alpha^*,\omega^k) \to 0$ for $k \to \infty$.

While many rounding algorithms are often much faster than the solution process for \eqref{eq:r},
there are situations, where long runtimes may occur. This is particularly the case for
switching cost aware rounding, which is motivated by the desire to reduce oscillations
in the resulting control functions and where a combinatorial optimization problem is solved for
every partition. For one-dimensional domains $\Omega$, this can be done 
with a shortest path algorithm \cite{bestehorn2021mixed}. While the best known runtime
estimate is in $O(N)$ with respect to $N$, it is also in $O(\sqrt{M}(2\lceil C\rceil + 1)^M)$ with respect
to $M$ \cite[Theorem 6.15]{bestehorn2021combinatorial},
which makes a small number of grid cells $N$ already
desirable if $M$, usually a parameter that cannot be changed, is of medium size, say, $M = 10$.
This effect is amplified if many similar optimization problems need to be solved  and the 
solutions of \eqref{eq:r} can be assumed to be close, for example in 
model predictive control or parameter identification.
Furthermore, switching cost aware rounding is much more costly on
multi-dimensional domains, see, \cite{bestehorn2021switching}, and
efficient combinatorial algorithms that are polynomial in $N$
for a fixed parameter $M$ have not been found so far. Thus the
combinatorial optimization problems arising from switching cost
aware rounding on multi-dimensional domains need to be solved
with black-box integer programming solvers or other 
enumeration algorithms that may have prohibitive runtimes.

We address these practical limitations of switching
cost aware rounding by proposing a non-uniform grid refinement
strategy for use in switching cost aware rounding. In particular,
we aim to reduce the number of variables for enumeration
algorithms so that larger problem instances can be handled in
practice.
\paragraph{Contribution}
We start from the observation that, depending on $\alpha^*$ and the
previously computed $\omega^k$, certain grid cells need not to be refined
in order to improve $d^k(\alpha^*,\omega^k)$.
We propose an algorithm that alternatingly splits grid cells and executes a 
computationally cheap rounding algorithm, which is well-defined for
non-uniform partitions of $\Omega$, until a (non-uniform) partition is found with
which a given desired accuracy in terms of $d^k(\alpha^*,\omega^k)$ can be reached.
Then the more expensive switching cost aware rounding may
be executed on the identified partition and a feasible point is guaranteed.
We prove that the proposed strategy always identifies a grid and a binary-valued
control function that meets the desired accuracy after finitely many iterations,
which in turn leads to the desired weak-$^*$ convergence $\omega^k \weakstarto \alpha^*$.

We define switching cost aware rounding on non-uniform grids as are generated
by the aforementioned algorithm and propose several improvements of the integer
programming formulation.

We apply our insights to a benchmark problem with two-dimensional domain from \cite{bestehorn2021switching}.
Our computational results show that the non-uniform grids and our improvements of the
integer programming formulation often have a beneficial effect in terms of tractability and remaining
duality gaps of the resulting mixed-integer programs.

\paragraph{Structure of the Remainder}
We provide the necessary definitions and notation in 
\Cref{sec:notation}.
We provide and analyze the grid refinement algorithm in
\Cref{sec:non-uniform_grid_refinement}. We present switching cost
aware rounding on non-uniform grids in \Cref{sec:scarp_non-uniform}.
Our computational experiments are summarized in \Cref{sec:computational_experiments}.

\section{Notation and Preparations}\label{sec:notation}
We abbreviate $[n] \coloneqq \{1,\ldots,n\}$ for $n \in \N$.
We write $\mathbbm{1}$ for the vector containing $1$ in
all entries, that is $\mathbbm{1} = (1,\ldots,1)^T$,
in Euclidean spaces. For a set $A$, we denote its
$\{0,1\}$-valued characteristic function by $\chi_A$.
We denote the canonical basis vectors of $\R^d$ for $d\in\N$
by $e_i$ for $i \in [d]$.

We briefly define the required concepts from the combinatorial integral decomposition
but refer to the referenced literature for details.
\begin{definition}[Definition 2.1 in \cite{manns2020multidimensional}]
A function $\alpha \in L^\infty(\Omega, \R^M)$ is a \emph{relaxed control}
if $\alpha(s) \ge 0$ for a.a.\ $s \in \Omega$ and $\sum_{m=1}^M \alpha_m(s) = 1$
for a.a.\ $s \in \Omega$. A relaxed control $\omega$ is a \emph{binary
control} if, in addition, $\omega(s) \in \{0,1\}^M$ for a.a.\ $s \in \Omega$.
\end{definition}
\begin{definition}[Definition 3.1 in \cite{kirches2021compactness}]
An ordered partition $\TT = (T_1,\ldots,T_N)$, $N \in \N$, of $\Omega$
with pairwise disjoint and measurable sets $T_n$ is a \emph{rounding grid}.
The quantity $\vol(\TT) \coloneqq \max_{n \in [N]} \lambda(T_n)$ is 
the \emph{grid constant of $\TT$}.
\end{definition}
\begin{proposition}[Lemma 4.1 in \cite{kirches2021compactness}]\label{prp:pseudometric}
Let $\TT$ be a \emph{rounding grid}. Then the function $d : L^\infty(\Omega,\R^M) \times L^\infty(\Omega,\R^M) \to [0,\infty)$ defined
as
\[ d^\TT(\alpha,\omega) \coloneqq \max_{n \in [N]} \left\|\int_{\bigcup_{i=1}^n T_i} \alpha(s) - \omega(s)\dd s\right\|_\infty \]
for all $\alpha$, $\omega \in L^\infty(\Omega,\R^M)$
is a pseudometric.
\end{proposition}
\begin{definition}\label{dfn:ra}
A rounding algorithm $\RA$ is a function that maps a relaxed control and
a rounding grid to a binary control. Formally,
\begin{multline*}
\RA : \{ 
\alpha \in L^\infty(\Omega,\R^M)\,|\,
\alpha \text{ is feasible for }\eqref{eq:r}\}
\times \{ \TT \,|\,
\TT \text{ is a rounding grid }\}\\
\to \{ 
\omega \in L^\infty(\Omega,\R^M)\,|\,
\omega \text{ is feasible for }\eqref{eq:p}\}
\end{multline*}
\end{definition}

\section{Non-uniform Rounding Grid Refinement}\label{sec:non-uniform_grid_refinement}
The main inputs for grid refinement algorithm introduced below are a relaxed control
$\alpha$ and a maximum acceptable error $\Delta > 0$. Then the algorithm 
computes a rounding grid $\TT$ such  that there is a binary control $\omega$, cellwise
constant on the computed rounding grid, which satisfies
\[ d^{\TT}(\alpha,\omega) \le \Delta. \]
Such desired combinations of input and output are formally defined below.
\begin{definition}\label{dfn:admissible_pair}
A tuple $(\alpha, \TT, \Delta)$ of a relaxed control $\alpha$,
a rounding grid $\TT$, and a positive scalar $\Delta > 0$
is \emph{admissible} if
\begin{enumerate}
\item\label{itm:integer_cell_multiples}
      there is $\underline{\vol} > 0$ such that for all $n \in [N]$
      there exists a natural number $k_n \in \N$ with
      $\lambda(T_n) = k_n \underline{\vol}$,
\item \label{itm:feasible_control_for_scarp}
      and there exists a binary control $\omega$
      that is constant on each grid cell
      $T_n$, $n \in \N$,
      such that $d^{\TT}(\alpha, \omega) \le \Delta$.
\end{enumerate}
\end{definition}
\Crefpart{dfn:admissible_pair}{itm:integer_cell_multiples} formalizes that
the volumes of the grid cells are all positive integer
multiples of a positive real scalar. \Crefpart{dfn:admissible_pair}{itm:feasible_control_for_scarp}
formalizes that we seek binary controls that satisfies
the desired estimates. This is the class of
rounding grids to which the efficient combinatorial algorithm
that perform switching cost aware rounding on one-dimensional domains can be generalized,
see \Cref{sec:scarp_non-uniform}.

We continue by generalizing the definition of sequences of refined
rounding grids and prove the desired weak-$^*$ convergence for the computed
binary controls in \Cref{sec:dissections}.
Then we provide the grid refinement algorithm in \Cref{sec:algo} and show that
it produces rounding grids that satisfies the assumptions in \Cref{sec:dissections}.
We prove its asymptotics in \Cref{sec:algo_proof}.

\subsection{$\alpha$-adapted Order-conserving Domain Dissections}
\label{sec:dissections}
The analysis in \cite{manns2020multidimensional,kirches2021compactness} is built
on the grid refinement assumption Definition 4.3 in \cite{manns2020multidimensional},
which requires that all grid cells are refined.
We provide a modified variant of this definition, which allows to omit refinements
if the relaxed control $\alpha$ is already binary-valued on parts of the domain.
\begin{definition}\label{dfn:order_conserving_domain_dissection}
Let $\alpha$ be a relaxed control. A sequence of rounding grids $(\TT^k)_k$ is
an \emph{$\alpha$-adapted order-conserving domain dissection} if
\begin{enumerate}
\item\label{itm:vanishing_max_cell_colume} $\max \left\{ \lambda(T_n^k)\,\middle|\, 
n \in [N^k] \text{ with } \lambda(\{s\,|\,\alpha(s) \notin \{0,1\}^M\} \cap T_n^k) > 0\right\}
\to 0$,
\item\label{itm:decomposition} for all $k \in \N \setminus\{0\}$ and all $n \in [N^{k-1}]$,
there exist $1 \le i \le j \le N^k$ such that
$T_n^{k-1} = \bigcup_{\ell = i}^j T_{\ell}^k$, and
\item\label{itm:regular_shrinkage} the cells $T^k_n$ shrink regularly: there is $C > 0$
such that for each $T^k_n$ there is a ball $B_n^k$
such that $T^k_n \subset B_n^k$ and
$\lambda(T^k_n) \ge C \lambda(B_n^k)$.
\end{enumerate}
\end{definition}
\Crefpart{dfn:order_conserving_domain_dissection}{itm:vanishing_max_cell_colume}
means that
the maximum volume of the grid cells on which $\alpha$ is not already a binary control vanishes over 
the grid iterations. This relaxes Definition 4.3 (2) from \cite{manns2020multidimensional}.
\Crefpart{dfn:order_conserving_domain_dissection}{itm:decomposition} means that a cell is
split into finitely many cells and the ordering of the grid cells from the preceding iterations
is preserved. \Crefpart{dfn:order_conserving_domain_dissection}{itm:regular_shrinkage} constrains
the eccentricity of the grid cells.

In order to be able to apply the arguments in 
\cite{manns2020improved,manns2020multidimensional,kirches2021compactness}
to $\alpha$-adapted order-conserving domain dissections, we  show 
$\omega^k \weakstarto \alpha$ in $L^\infty(\Omega,\R^M)$ if $d^{\TT^k}(\alpha,\omega^{k}) \to 0$ for an $\alpha$-adapted order conserving
domain dissection $(\TT^k)_k$ and a sequence of binary controls $(\omega^k)_k$ that coincide with $\alpha$ on the cells that are excluded in \Crefpart{dfn:order_conserving_domain_dissection}{itm:vanishing_max_cell_colume}.
\begin{lemma}\label{lem:weakstar}
Let $\alpha$ be a relaxed control and $(\TT^k)_k$ an $\alpha$-adapted order-conserving
domain dissection. Let $(\phi^k)_k \subset L^\infty(\Omega,\R^M)$ satisfy
\begin{enumerate}
\item $\phi^k(s) \in [-1,1]^M$ a.e.\ in $\Omega$ for all $k \in \N$,
\item\label{itm:phi_zero_where_alpha_zero} $\phi^k(s) = 0$ a.e.\ in $T^k_i$ if $\alpha(s) = 0$
      a.e.\ in $T^k_i$ for all $i \in [N^k]$, $k \in \N$, and
\item\label{itm:pseudometric_to_zero} $\max \Big\{
    \Big\|\int_{\bigcup_{i=1}^n T_i^k} \phi^k(s)\dd s\Big\|_\infty \,\Big|\,
    n \in [N^k] \Big\} \to 0$ for $k \to \infty$.
\end{enumerate}
Then $\int_{\Omega} \phi^k_i(s)f(s)\dd s \to 0$ for $k \to \infty$
holds for all $f \in L^1(\Omega)$ and all $i \in [M]$.
\end{lemma}
\begin{proof}
We adapt the proof of \cite[Lemma 4.4]{manns2020multidimensional} and start by fixing
$i \in \{1,\ldots,M\}$
and writing $\phi^k$ instead of $\phi^k_i$ to reduce notational bloat. As in \cite[Lemma 4.4]{manns2020multidimensional},
we observe that the functions $\phi^k f$ are integrable and to restrict to $f \ge 0$ a.e.

Let $\varepsilon > 0$. As in\cite[Lemma 4.4]{manns2020multidimensional}, there exists a
non-negative simple function such that for all $k \in \N$:
\begin{gather}\label{eq:simple_function_approx}
0 \le \tilde{f}(s) \le f(s) \text{ a.e.\ and }
\left|\int_{\Omega} \phi^k(s)f(s)\dd s\right| \le
\left|\int_{\Omega}\phi^k(s) \tilde{f}(s)\dd s\right| + \frac{\varepsilon}{3}.
\end{gather}

For $k \in \N$ we define the function $\tilde{f}^k$ for $s \in \Omega$ as
\[ \tilde{f}^k(s) \coloneqq \sum_{i=1}^{N^k} \tilde{f}^{k}_i(s)
\text{ with }
\tilde{f}^k_i(s) \coloneqq \left\{\begin{aligned}
\tilde{f}(s) & \text{ if } \alpha = 0 \text{ a.e.\ in } T_i^k,\\
\frac{\chi_{T_i^k}(s)}{\lambda(T^k_i)}\int_{T^k_i} \tilde{f}(\sigma) \dd \sigma & \text{ else}
\end{aligned}\right.
\]
Then $\tilde{f}^k \to \tilde{f}$ pointwise a.e.\ by virtue of Lebesgue's differentiation theorem,
which may be implied because we impose the regular shrinkage condition on the grid cells of
$\alpha$-adapted order-conserving domain dissections that are refined infinitely often
in \Crefpart{dfn:order_conserving_domain_dissection}{itm:regular_shrinkage} and
$\tilde{f}^k$ coincides with $\tilde{f}$ on cells that are not refined infinitely often. Moreover,
$0 \le \tilde{f}^k(s) \le \|\tilde{f}\|_{L^\infty(\Omega)}$ a.e.\ by construction through averaging and $\|\tilde{f}\|_{L^\infty(\Omega)} < \infty$
because $\tilde{f}$ is a simple function. Thus Lebesgue's dominated convergence theorem implies
$\tilde{f}^k \to \tilde{f}$ in $L^1(\Omega)$, implying that there is $k_0 \in \N$ such that
\begin{gather}\label{eq:lebesgue_differentiation_approx}
    \|\tilde{f} - \tilde{f}^{k_0}\|_{L^1(\Omega)} \le \frac{\varepsilon}{3}.
\end{gather}

Let $k \in \N$. Then we obtain
\[ \left|\int_{\Omega} \phi^k(s) f(s)\dd s\right|
\underset{\mathclap{\eqref{eq:simple_function_approx}}}\le
\left|\int_{\Omega}\phi^k(s) \tilde{f}(s)\dd s\right| + \frac{\varepsilon}{3}
\le \left|\int_{\Omega}\phi^k(s) \tilde{f}^{k_0}(s)\dd s\right|
+ \frac{2\varepsilon}{3},
\]
where the second inequality follows from the triangle inequality and
H\"{o}lder's inequality with $\|\phi^k\|_{L^\infty(\Omega)} \le 1$
and the estimate \eqref{eq:lebesgue_differentiation_approx}.
It remains to prove that
\[ \left|\int_{\Omega}\phi^k(s) \tilde{f}^{k_0}(s)\dd s\right| < \frac{\varepsilon}{3} \]
for all $k \ge k_1$ for some $k_1 \in \N$ large enough. We prove this assertion
by showing $\int_{T_i^{k_0}} \phi^k(s) \tilde{f}^{k_0}(s)\dd s \to 0$
for $k \to \infty$ for all $i \in [N^{k_0}]$.
To this end, we distinguish two cases on $T_i^{k_0}$.

\textbf{Case $\alpha(s) = 0$ a.e.\ in $T_i^{k_0}$:} Then
$\phi^{k_0}(s) = 0$ a.e.\ in $T_i^{k_0}$ because of assumption \ref{itm:phi_zero_where_alpha_zero}.
Combining this with the recursive refinement principle from 
\Crefpart{dfn:order_conserving_domain_dissection}{itm:decomposition}, we obtain
$\phi^{k}(s) = 0$ a.e.\ in $T_i^{k_0}$ for all $k \ge k_0$. Consequently,
$\int_{T_i^{k_0}} \phi^k(s) \tilde{f}^{k_0}(s)\dd s = 0$
if $\alpha(s) = 0$ a.e.\ in $T_i^{k_0}$.

\textbf{Case $\lambda(\{s\,|\,\alpha(s) > 0\} \cap T_i^{k_0}) > 0$:} 
Then
\begin{gather}\label{eq:split_integral}
\begin{aligned}
\int_{T_i^{k_0}} \phi^k(s) \tilde{f}^{k_0}(s)\dd s
&= \tilde{f}^{k_0}_i \int_{T_i^{k_0}} \phi^k(s)\dd s\\
&= \tilde{f}^{k_0}_i \left(\int_{\bigcup_{j=1}^i T_j^{k_0}} \phi^k(s)\dd s
- \int_{\bigcup_{j=1}^{i-1} T_j^{k_0}} \phi^k(s)\dd s
\right),
\end{aligned}
\end{gather}
where the first identity follows by construction of $\tilde{f}^{k_0}$.
\Crefpart{dfn:order_conserving_domain_dissection}{itm:decomposition} gives that
for all $k \ge k_0$ there exist indices $a(i,k)$, $b(i,k) \in [N^k]$ such that
$T_i^{k_0} = \bigcup_{\ell=a(i,k)}^{b(i,k)} T_\ell^k$, which implies
\[ \int_{T_i^{k_0}} \phi^k(s) \tilde{f}^{k_0}(s)\dd s
= \tilde{f}^{k_0}_i \left(\int_{\bigcup_{j=1}^{b(i,k)} T_j^{k}} \phi^k(s)\dd s
- \int_{\bigcup_{j=1}^{a(i,k)} T_j^{k}} \phi^k(s)\dd s
\right)
\]
when inserted into \eqref{eq:split_integral}.
Assumption \ref{itm:pseudometric_to_zero} implies
$\int_{\bigcup_{j=1}^{b(i,k)} T_j^{k}} \phi^k \to 0$
and $\int_{\bigcup_{j=1}^{a(i,k)} T_j^{k}} \phi^k \to 0$,
which in turn gives 
implying $\int_{T_i^{k_0}} \phi^k) \tilde{f}^{k_0} \to 0$ as desired.

Because $k_0$ is fixed now, an $\tfrac{\varepsilon}{3 N^{k_0}}$ argument
yields the claim $\left|\int_{\Omega} \phi^k \tilde{f}^{k_0}\right| < \varepsilon$.
This closes the proof.
\end{proof}
\begin{corollary}\label{cor:weakstar}
Let $\alpha$ be a relaxed control.
Let $d^{\TT^k}(\alpha,\omega^{k}) \to 0$ for an $\alpha$-adapted order conserving
domain dissection $(\TT^k)_k$ and a sequence of binary controls $(\omega^k)_k$ that
coincide with $\alpha$ on the cells that are excluded in 
\Crefpart{dfn:order_conserving_domain_dissection}{itm:vanishing_max_cell_colume}.
Then $\omega^k \weakstarto \alpha$.
\end{corollary}
\begin{proof}
This follows with the choice $\phi^k \coloneqq \alpha - \omega^k$ for $k \in \N$.
\end{proof}

\subsection{Algorithm Statement and Description}\label{sec:algo}

The grid refinement algorithm is stated formally
in \Cref{alg:adaptive_grid_refinement}. It has five inputs:
(1) a relaxed control $\alpha$, (2) a tolerance for the difference between
$\alpha$ and its binary control approximation $\omega$ in the pseudometrics
$d^\TT$, (4) an initial rounding grid $\TT^0$ discretizing the domain $\Omega$,
a constant $C > 0$ so that the grid cells in $\TT^0$ satisfy the regularity condition
\Crefpart{dfn:order_conserving_domain_dissection}{itm:regular_shrinkage} with $C$,
and (5) a rounding algorithm $\RA$.

\Cref{alg:adaptive_grid_refinement} proceeds as follows. In \Cref{ln:condition}, it
executes the rounding algorithm $\RA$ on $\alpha$ and the current rounding grid $\TT^k$.
Then the distance between $\alpha$ and the output of $\RA$ is measured with $d^{\TT^k}$.
When the distance falls below the tolerance $\Delta$, the algorithm
terminates and returns $\TT^k$ because a binary control that is
close enough to $\alpha$ with respect to $d^{\TT^k}$ has been found.
If the termination test in \Cref{ln:condition} is not successful, the algorithm
refines the rounding grid: it identifies a grid cell $n^* \in [N^k]$ that
maximizes a measure of the \emph{non-binariness} of the relaxed control $\alpha$
weighted by the volume of the grid cell in \Cref{ln:identify_most_fractional_cell},
where the $\min\{\cdot,\cdot\}$ is understood componentwise. Then the identified
grid cell is split into finitely many cells, each having a volume less than half of the
split cell. We note that the $\arg \max$ in \Cref{ln:identify_most_fractional_cell}
may not be unique and we assume that an arbitrary choice is made in this case.

\begin{algorithm}[h]
\caption{Rounding Grid Refinement}\label{alg:adaptive_grid_refinement}
\SetAlgoLined
\KwData{Relaxed control $\alpha$,
tolerance $\Delta \ge 0$, initial rounding grid $\TT^0$ 
with $N^0$ grid cells of $\Omega$, $C > 0$
such that the grid cells $T_n^0$, $n \in [N^0]$, in $\TT^0$ satisfy
\Crefpart{dfn:admissible_pair}{itm:integer_cell_multiples} and
\Crefpart{dfn:order_conserving_domain_dissection}{itm:regular_shrinkage}.}
\KwData{Rounding algorithm $\RA$.}
$k\gets 0$\\
\label{ln:condition}
\While{$d^{\TT^k}\left(\alpha,\RA(\alpha, \TT^k)\right) > \Delta$}{
  $n^* \gets \arg \max\left\{
  \left\|\min \left\{ \int_{T_n^k} \alpha(s)\dd s,
  \int_{T_n^k} \mathbbm{1} - \alpha(s)\dd s
  \right\}
  \right\|_\infty
  \,\Big|\, n \in [N^k]
  \right\}$\label{ln:identify_most_fractional_cell}\\
  Compute a disjoint partition $\hat{T}_1^{k} \cup \cdots \cup \hat{T}^k_{\ell^{k}} = T^k_{n^*}$, 
  $\ell^k \in \N \setminus \{0,1\}$, of measurable sets $\hat{T}_j^{k}$ such that there are balls 
  $B_j^k$ satisfying $\hat{T}_j^k \subset B_j^k$ and $\lambda(\hat{T}^k_j) \ge C \lambda(B^k_j)$
  and $\lambda(\hat{T}^k_j) = \frac{1}{a_j} \lambda(T^k_{n^*})$ with $a_j \in \N \setminus \{0,1\}$
  \label{ln:fractional_cell_split}\\
  $\TT^{k+1} \gets \left(T_1^k,\ldots,T_{n^* - 1}^k,\hat{T}_1^{k},\ldots,\hat{T}^k_{\ell^{k}},
  T_{n^* + 1}^k,\ldots,T_{N^k}\right)$\\
  $N^{k+1} \gets N^k + \ell^k - 1$\\
  $k\gets k + 1$
}
\KwResult{Final rounding grid $\TT^k$ of $\TT$.}
\end{algorithm}

We provide some details on the identification of $n^* \in [N^k]$ in 
\Cref{ln:identify_most_fractional_cell}. We seek binary controls
$\omega$ that coincide with $\alpha$ in cells $\TT^k_n$, where $\alpha$ is
already a binary control. In this case, we obtain for each $m \in [M]$ that
\[ \frac{1}{\lambda(\TT^k_n)}\left|\int_{\TT^k_n} \alpha_m(s) - \omega_m(s)\dd s\right|
 = \frac{1}{\lambda(\TT^k_n)}\left\{ 
\begin{aligned}
\int_{\TT^k_n} \alpha_m(s)\dd s & \text{ if } \omega_m = 0 \text{ a.e.\ in } \TT^k_n,\\
\int_{\TT^k_n} 1 - \alpha_m(s)\dd s & \text{ else.}
\end{aligned}
\right.
\]
Thus the minimum of these two values may be used as an indicator on how close the $m$-th
component of $\alpha$ is to a cell-wise constant binary control on $\TT^k_n$.
Weighting this indicator with the volume of $\TT^k_n$, we obtain
$\min\big\{\int_{\TT^k_n}\alpha_m(s)\dd s, \int_{\TT^k_n} 1-\alpha_m(s)\dd s\big\}$
as a measure of non-binariness to assess the contribution to the $m$-th component on the
interval $\TT^k_n$ in
the integral $\int_{\bigcup_{i=1}^j T^k_i} \alpha(s) - \omega(s)\dd s$ for $j \ge n$.
As in the definition of $d$ we take the $\max$-norm of the resulting vector
$\min\big\{ \int_{\TT^k_n} \alpha(s)\dd s, \int_{\TT^k_n} \mathbbm{1} - \alpha(s)\dd s\big\}$
in order to obtain a scalar-valued measure of the non-binariness for
each grid cell.

Before we prove that \Cref{alg:adaptive_grid_refinement} produces
a rounding grid $\TT^k$ after finitely many iterations such that
the tuple $(\alpha,\TT^k,\Delta)$ is admissible for given $\alpha$ and $\Delta$,
we show that the produced sequence of rounding grids
is well defined: it complies with 
\Crefpart{dfn:admissible_pair}{itm:integer_cell_multiples} and
\Crefparts{dfn:order_conserving_domain_dissection}{itm:decomposition}{itm:regular_shrinkage} in all iterations.
\begin{lemma}\label{lem:rounding_grids_well_defined}
Let $\alpha$ be a relaxed control. Let $\Delta \ge 0$.
Let $\TT^0$ be a rounding grid with $C > 0$ such that all
grid cells $T_n^0$, $n \in [N^0]$, satisfy
\Crefpart{dfn:admissible_pair}{itm:integer_cell_multiples} and
\Crefpart{dfn:order_conserving_domain_dissection}{itm:regular_shrinkage}.
Let $\RA$ be a rounding algorithm.
Then \Cref{alg:adaptive_grid_refinement} produces a, possibly infinite,
sequence of rounding grids $\TT^k$ that satisfy
\Crefpart{dfn:admissible_pair}{itm:integer_cell_multiples} and
\Crefparts{dfn:order_conserving_domain_dissection}{itm:decomposition}{itm:regular_shrinkage} for all $k \in \N$.
\end{lemma}
\begin{proof}
The claim is shown inductively. The base case holds by assumption.
We assume that the claim holds for $\TT^k$, $k \in \N$.
In iteration $k + 1$, there is one grid cell $T^k_{n^*}$ that is split into
$\ell^k \ge 2$  grid cells that are inserted into $\TT^k$ between $T^k_{n^*-1}$ and $T^k_{n^*+1}$.
This implies that the recursive ordering property \Crefpart{dfn:order_conserving_domain_dissection}{itm:decomposition} holds for grid cell $n^*$ from iteration $k$.
For $n < n^*$, we obtain $T^{k}_n = T^{k+1}_n$ and for $n > n^*$, we obtain
$T^{k}_n = T^{k+1}_{n + \ell^k}$. 
By virtue of \Cref{alg:adaptive_grid_refinement} \Cref{ln:fractional_cell_split},
\Crefpart{dfn:order_conserving_domain_dissection}{itm:regular_shrinkage}
is also satisfied for the newly added cells with the same constant $C > 0$ (the
input of the algorithm).
\Cref{alg:adaptive_grid_refinement} \Cref{ln:fractional_cell_split} also yields
$\lambda(T^{k+1}_{n^* - 1 + j}) = \frac{1}{a_j} \lambda(T^k_{n^*})$ with $a_j \in \N \setminus \{0,1\}$
for the newly added cells $T^{k+1}_{n^* - 1 + j}$, $j \in [\ell^k]$.
By induction we have $\lambda(T^k_{n}) = k_n \underline{\vol}$ holds for some $\underline{\vol} > 0$
and $\{k_n\,|\,n \in [N^k]\} \subset \N$. 
We set $\underline{\vol}^* \coloneqq \underline{\vol} / \Pi_{j=1}^{\ell^k} a_j$,
$k_n^* \coloneqq k_n \Pi_{j=1}^{\ell^k} a_j$ for all $n \in [n^*]$ and all
$n \in [N^{k+1}]\setminus [n^* + \ell^k - 1]$, which implies
$k_n^* \underline{\vol}^* = k_n \underline{\vol} = \lambda(T^{k+1}_{n})$.
Furthermore, we set $k_n^* \coloneqq k_{n^*} \Pi_{j=1, j \neq n - (n^* - 1)}^{\ell^k} a_j$
for all $n \in \{n^*,\ldots,n^* + \ell^k - 1\}$, which implies
$k_n^* \underline{\vol}^* = k_n \underline{\vol} / a_{n - (n^* - 1)} = \lambda(T^{k+1}_{n})$.
In total, this construction gives $k_n^* \in \N$
and $k_n^* \underline{\vol}^* = \lambda(T^{k+1}_{n})$ for all $n \in [N^{k+1}]$, which
implies \Crefpart{dfn:admissible_pair}{itm:integer_cell_multiples}.
This completes the induction step and hence the proof.
\end{proof}

\subsection{Proof of Asymptotics of \Cref{alg:adaptive_grid_refinement}}\label{sec:algo_proof}

We prove that the input and outputs of 
\Cref{alg:adaptive_grid_refinement} are admissible
in the sense of \Cref{dfn:admissible_pair}
for the case that a variant of the sum-up rounding algorithm 
from \cite{manns2021approximation}
is chosen for $\RA$.

We introduce the sum-up rounding variant briefly
as \Cref{alg:sur} below.
\Cref{alg:sur} computes a binary control $\omega$ that is cell-wise constant on the
$T_n$. It iterates over the $T_n$ from $n = 1$ to $N$ and computes vectors
$\hat{\omega}_n \in \R^M$, which are then the values of $\omega$ the $T_n$.
In each iteration $n$ (grid cell $T_n$), \Cref{alg:sur} computes
the sum of the integral over the difference between
the relaxed control $\alpha$ and $\omega$ from $T_0$ to $T_{n-1}$ plus the integral over
alpha over the current grid cell $T_n$. If $\alpha_i(s) = 1$ for a.a.\ $s \in T_n$ and some
$i \in [M]$ (there is at most one), then $\omega_i$ is set to $1$ on $T_n$. Else, a maximizing
index $i$, which can be chosen arbitrary in case of non-uniqueness, of this vector
$\gamma_n \in \R^M$ is selected. Then $\omega_i$ is set to $1$ on $T_n$ and $\omega_j$ is set
to zero on $T_n$ for $j\neq i$ by setting $\hat{\omega}_n \gets e_i$, thereby ensuring
that the resulting function $\omega$ is a binary control.
The algorithm differs from the standard version of sum-up rounding
in that $\omega$ is always set to $\alpha$ on intervals, where
$\alpha$ is already a.e.\ binary-valued.

\begin{algorithm}[h]
\caption{Variant of the Sum-Up Rounding Algorithm}\label{alg:sur}
\SetAlgoLined
\KwData{Relaxed control $\alpha$, rounding grid $\TT = (T_1,\ldots,T_N)$ of $\Omega$.}
$\phi_0 \gets 0_{\R^M}$\\
\For{$n = 1,\ldots,N$}{
    $\gamma_{n} \gets \phi_{n-1} + \int_{T_n} \alpha(s)\dd s$\\
    $i \gets \left\{
    \begin{aligned}
    & j \text{ if there exists } j \text{ such that } \int_{T_n} \alpha_j(s)\dd s = \lambda(T_n),\\
    & \arg \max \{ \gamma_{n,j}\,|\, j \in [M]\} \text{ else}.
    \end{aligned}\right.$\\
    $\hat{\omega}_{n} \gets e_i$\\
    $\phi_{n} \gets \gamma_n - \hat{\omega}_{n}\lambda(T_n)$
}
\KwResult{Final binary control $\omega = \sum_{n=1}^N
\chi_{T_n} \hat{\omega}_{n}$.}
\end{algorithm}

\begin{theorem}\label{thm:convergence_w_variant_of_sur}
Let $\alpha$ be a relaxed control. Let $\Delta > 0$. Let $\TT^0$ be a rounding grid
with $C > 0$ such that all grid cells $T_n^0$, $n \in [N^0]$, satisfy
\Crefpart{dfn:admissible_pair}{itm:integer_cell_multiples} and
\Crefpart{dfn:order_conserving_domain_dissection}{itm:regular_shrinkage}.
Let $\RA = \Cref{alg:sur}$. Then \Cref{alg:adaptive_grid_refinement} terminates
after finitely many iterations
such that $\alpha$, $\Delta$, and the rounding grid 
$\TT^k$ produced in the final iteration are
admissible in the sense of \Cref{dfn:admissible_pair}.
\end{theorem}
\begin{proof}
\Cref{lem:rounding_grids_well_defined} and the termination criterion of
\Cref{alg:adaptive_grid_refinement} in \Cref{ln:condition} give that the triple
$(\alpha, \Delta, \TT^{k_f})$ for the final iteration $k_f \in \N$ is admissible if
\Cref{alg:adaptive_grid_refinement} terminates after finitely many iterations.
Thus it remains to show that the algorithm terminates after
finitely many iterations.

We note that if iteration $k$ is executed, there is some $n \in [N^k]$ such that
\[ \delta_n^k \coloneqq \left\|
  \min\left\{ \int_{T_n} \alpha(s)\dd s, \int_{T_n} \mathbbm{1} - \alpha(s)\dd s\right\}
  \right\|_\infty > 0.
\]
If this were not the case, then $\alpha$ would be a binary control and
\Cref{alg:sur} would have terminated with resulting binary control $\alpha$, implying
$d^{\TT^k}(\alpha, \alpha) = 0 < \Delta$, which in turn would have
caused termination of \Cref{alg:adaptive_grid_refinement}.

We prove the claim by assuming that \Cref{alg:adaptive_grid_refinement}
does not terminate after finitely many iterations and obtaining a contradiction.
If \Cref{alg:adaptive_grid_refinement} does not terminate after finitely many iterations,
infinitely many grid cells $T^k_{n^*}$ are split. Because $\Omega$ is bounded,
there is an infinite subsequence of the iterations, indexed by $\ell$, such that the
grid cells $((n^*)^{k_\ell})_\ell$ selected in
\Cref{ln:identify_most_fractional_cell} satisfy
$T^{k_\ell}_{{n^*}^{k_\ell}} \supset T^{k_{\ell + 1}}_{{n^*}^k_{\ell + 1}}$.
Moreover, we obtain
\[ \lambda(T^{k_{\ell + 1}}_{{n^*}^k_{\ell + 1}})
   \le \frac{1}{2} \lambda(T^{k_{\ell}}_{{n^*}^k_{\ell}})
   \le \ldots \le \frac{1}{2^\ell} \lambda(\Omega)
\]
because the volumes are only a fraction less than $0.5$ of the split
grid cell, see \Cref{alg:adaptive_grid_refinement} \Cref{ln:fractional_cell_split}.
Clearly, the definition of $\delta^k_n$ implies that $\delta^k_n \le \lambda(T^{k}_{n})$,
implying that $\delta^{k_\ell}_{{n^*}^{k_\ell}}$, the maximum value in
\Cref{ln:identify_most_fractional_cell} in iteration $k_\ell $ is bounded by
$\frac{1}{2^\ell} \lambda(\Omega) \to 0$ for $\ell \to \infty$. Because
$\max \{ \delta_n^k \,|\, n \in [N^k]\}$ is monotonously non-increasing from one iteration
of \Cref{alg:adaptive_grid_refinement} to the next, this implies that
$\max \{ \delta_n^k \,|\, n \in [N^k]\} \to 0$ for $k \to \infty$.

Next, we show that this implies that
\begin{gather}\label{eq:ivals_to_zero}
\Delta^k \coloneqq \max\left\{\lambda(T_n^k) \,\middle|\,
\begin{aligned}
\bigg\|\min\bigg\{\int_{T_n^k} \alpha(s)\dd s,
\int_{T_n^k} \mathbbm{1} - \alpha(s)\dd s\bigg\}
\bigg\|_\infty > 0\\    
 \text{ for } n \in [N^k]
\end{aligned}\right\} \to 0
\end{gather}
for $k \to \infty$. Assume this were not the case. Because grid cells
are either left unchanged or split into smaller cells whose volume
is only a fraction less than $0.5$ of the split grid cell, there exists a grid cell in
the sequence of rounding grids, indexed by $(n^k)_k$, which is not refined after finitely
many iterations and which satisfies $\delta^{k}_{n^k} > 0$,
implying that $\delta^{k_0}_{n^k_0} = \delta^k_{n^k}$
for all $k \ge k_0$ and some $k_0$. This can only
happen if $\delta^{k_0}_{n^k_0} = 0$ because
$\max \{ \delta_n^k \,|\, n \in [N^k]\} \to 0$,
which gives a contradiction and thus 
\eqref{eq:ivals_to_zero} holds true.

Let $k \in \N$, $n \in\N$. We observe that if $\alpha$ is binary-valued on a grid
cell $T_n^k$, then two properties hold. First, we observe that this is equivalent
to
\begin{gather*}
    \min \bigg\{
    \int_{T_n^k} \alpha(s)\dd s,  
    \int_{T_n^k} \mathbbm{1} - \alpha(s)\dd s \bigg\} = 0,
\end{gather*} 
implying that such grid cells are exactly those that are excluded from the
maximization in \eqref{eq:ivals_to_zero}. Second, the choice that $\omega$,
computed by \Cref{alg:sur}, on this grid is equal to $\alpha$ on $T_n^k$
implies $\phi_{n} = \phi_{n-1}$. Thus \Cref{alg:sur} behaves exactly as
if it were executed on a rounding grid where the grids $T_n^k$ are dropped.
Moreover, these grid cells do not alter to the quantity
$d^{\TT^k}(\alpha,\omega)$ because they satisfy the identity
$\int_{T^k_n} (\alpha - \omega) = 0$.

Assume that $\omega^k$ is computed by $\RA = \Cref{alg:sur}$
on rounding grid $\TT^k$. 
Because \Cref{alg:sur} behaves as if it were executed on a
grid without the intervals, on which $\alpha$ is binary-valued,
and the value of  $d^{\TT^k}(\alpha,\omega)$ is not altered on
these intervals do, we can apply the theory of the 
unmodified variant of sum-up rounding to the rounding
grid where these intervals are dropped.

This yields \cite{kirches2020approximation,sager2012integer,manns2020multidimensional}
that there exists some $\kappa > 0$, independent of $\TT^k$
and rounding grid $\alpha$, such that
\begin{gather}\label{eq:constant_c_sur}
d^{\TT^k}(\alpha,\omega^k) \le \kappa \Delta^k.
\end{gather}
Combining \eqref{eq:constant_c_sur} with \eqref{eq:ivals_to_zero}
implies $d^{\TT^k}(\alpha, \omega^k) \to 0$ 
and in turn that the termination criterion of 
\Cref{alg:adaptive_grid_refinement} is satisfied after finitely
many iterations.
\end{proof}

\section{Switching Cost Aware Rounding on Non-uniform Discretizations}
\label{sec:scarp_non-uniform}
We briefly introduce the switching cost aware rounding problem in its integer programming
formulation and describe its structure and known features. We point out why the problem
is considerably more intricate for multi-dimensional grids.
Afterwards, we provide a set
of linear inequalities that are valid cuts for the integer programming formulation.
Moreover, we provide a primal heuristic that is based on the algorithm that is used
for the one-dimensional case in \cite{bestehorn2021mixed}.

\subsection{Problem Formulation}

Based on a rounding grid $\TT = (T_1,\ldots,T_N)$, switching cost aware rounding
is a rounding algorithm by solving a mixed-integer program to determine a binary
control $\omega$ feasible for~\eqref{eq:p} while satisfying $d^{\TT}(\alpha,\omega) \le \Delta$.
To this end, we introduce variables $\omega_{n, m} \in \{0, 1\}$ for all $n \in [N]$ and $m \in [M]$,
yielding the control
\begin{equation*}
    \omega(s) \coloneqq \sum_{n=1}^N \left( \chi_{T_n}(s) \sum_{m=1}^M \omega_{n, m} e_m \right).
\end{equation*}

\paragraph{Constraint set}
To ensure feasibility, we impose the constraint $\sum_{m=1}^{M} \omega_{n, m} = 1$
for all $n \in [N]$ in order to model the one-hot encoding. To satisfy the distance requirement, we let 
\begin{equation*}
    \alpha_{n, m} \coloneqq \frac{1}{\lambda(T_n)} \int_{T_n} \alpha_m(s) \dd s
\end{equation*}
be the average value of $\alpha_m$ over $T_n$ and require that
\begin{gather}
    \label{eq:disc_distance}
    -\Delta \leq \sum_{i=1}^{n} \lambda(T_i) \left[ \omega_{i, m} - \alpha_{i, m} \right] \leq \Delta
\end{gather}
for all $n \in [N]$ and $m \in [M]$. Each one of these constraints
corresponds to a knapsack constraint
(see~\cite{pisinger1998knapsack} for a survey on knapsack problems)
either in the  variables $\omega$ or in their negations. Thus,
optimizing a linear objective in $\omega$ subjecting to just a single
constraint in \eqref{eq:disc_distance} is NP-hard in general.

In the following we include additional structure to the problem by making
another assumption regarding the rounding grid $\TT$.
Recall that the grid is obtained by a series of refinement steps
conducted by \Cref{alg:adaptive_grid_refinement} based on an initial grid $\TT^{0}$.
We assume that $\TT^{0}$ is composed of a single cell $T^{0}$ and that in each iteration
the disjoint partition $\hat{T}_1^{k} \cup \cdots \cup \hat{T}^k_{\ell^{k}} = T^k_{n^*}$ is such
that $\ell^{k}$ is a power of two and the sets $\hat{T}_j^{k}$
have the same volume of $1/{\ell^{k}} \lambda(T^k_{n^*})$ for 
all $j \in [\ell^{k}]$. In practice, such a partition can be achieved 
based on bisections along all $d$ axes. Consequently the volumes $\lambda(T_i)$
have values of $2^{-j_i} \lambda(T^{0})$,
where the values $j_i \in \N$ are determined by how many times $T^{0}$ is refined
by \Cref{alg:adaptive_grid_refinement}. We can therefore
scale \eqref{eq:disc_distance}
to obtain linear constraints of the form
\begin{equation}
    \label{eq:knapsack_inequalities}
    l_{m, n} \leq \sum_{i=1}^{n} 2^{k_i} \omega_{i, m} \leq u_{m, n},
\end{equation}
where $k_i \coloneqq j_{\max} - j_i$ for $j_{\max} \coloneqq \max_{n=1}^{N} j_n$,
\begin{equation*}
    \begin{aligned}
        l_{m,n} &\coloneqq
        \left\lceil
        2^{j_{\max}}
        \left( 
        - \frac{\Delta}{\lambda(T^{0})} + \sum_{i=1}^{n} 2^{-j_{i}} \alpha_{i,m}
        \right)
        \right\rceil
        \text{, and } \\
        u_{m,n} &\coloneqq 
        \left\lfloor 2^{j_{\max}}
        \left( 
        \frac{\Delta}{\lambda(T^{0})} + \sum_{i=1}^{n} 2^{-j_{i}} \alpha_{i,m}
        \right)
        \right\rfloor    
    \end{aligned}
\end{equation*}
for all $n \in [N]$ and $m \in [M]$. Note that the values $k_i$ are non-negative
with at least one being zero. The final set of constraints includes 
the distance requirements as well as SOS1 constraints
modeling the one-hot encoding of $\omega$:
\begin{equation}
    \label{eq:ip_cons}
    \begin{aligned}
        l_{m, n} \leq \sum_{i=1}^{n} 2^{k_i} \omega_{i, m} \leq u_{m, n} \quad & \forall m \in [M], n \in [N]\enskip\text{and} \\
        \sum_{j = 1}^{N} \omega_{i, j} = 1  \quad & \forall i \in [M]
    \end{aligned}
\end{equation}

Regarding the constraints~\eqref{eq:knapsack_inequalities}, note that
since $2^{k_i}$ divides $2^{k_j}$ for $k_i \leq k_j$,
we can order the variables ascending by $k_i$, thereby
ensuring that the weight of a given item divides the weights of all succeeding
items. A knapsack constraint with this property is called 
\emph{sequential}~\cite{hojny2020knapsack}. Sequential knapsack problems
can in fact be solved polynomial time~\cite{hartmann1993solving}.
What is more, a complete description of the sequential knapsack polytope 
is known~\cite{pochet1998sequential}. Unfortunately however,
the description is exponential in size and 
it is (to the best of our knowledge) not known how
to separate over it in polynomial time.

\paragraph{Objective} For every grid cell indexed by $n \in [N]$, there is a set $\mathcal{N}(n)$
of indices of adjacent (or neighboring) grid cells. The objective of switching cost aware rounding
can then be stated as
\begin{gather}\label{eq:scarp_objective}
\frac{1}{2} \sum_{n=1}^N \sum_{k \in \mathcal{N}(T_n)} 
  \sum_{i=1}^M \sum_{j=1}^M |\omega_{n,i} - \omega_{k,j}| 
  \Ha^{d-1}(\overline{T_n} \cap \overline{T_k})
  c_{ij} 
\end{gather}
for a symmetric cost coefficient matrix $(c_{ij})_{i,j \in [M]} \in \R^{M \times M}_{> 0}$,
where $c_{ij}$ models the cost of switching from control realization $i$ to $j$ (or vice versa)
from one grid cell to an adjacent one and
$\Ha^{d-1}(\overline{T_n} \cap \overline{T_k})$ is the $(d - 1)$-dimensional Hausdorff measure
of the $(d - 1)$-dimensional interface area that separates the grid cells $T_n$ and $T_k$
from each other.

\subsection{Valid Inequalities}
\label{sec:valid_ineqaulities}

We proceed to derive several classes of valid inequalities based on 
the constraints~\eqref{eq:knapsack_inequalities}. Since individual knapsack inequalities 
are well understood and separation algorithms are part of many state-of-the-art integer
programming solvers, we focus on combinations of several inequalities, exploiting the 
structure of the weights.

\paragraph{Lattice-based inequalities}

Consider fixed indices $n \in [N]$, $m \in [M]$ and let $k_{\max} \coloneqq \max_{i=1}^{n} k_i$. 
We split the set of variables, indexed by $I \coloneqq [n]$ into 
$I_{\max} \coloneqq \{i \in I \mid k_i = k_{\max}\}$ and $I_{C} \coloneqq I \setminus I_{\max}$.
The variables in $I_{\max}$ produce points on the lattice $\{j 2^{k_{\max}} \mid j \in \mathbb{Z} \}$ in
inequality~\eqref{eq:knapsack_inequalities}. The given lattice point is then shifted
depending on the variables in $I_{C}$. If we fix these variables to some 0/1
values $\overline{\omega}_{I_{C}, m}$, we obtain the following:
\begin{equation*}
    l(\overline{\omega}) \coloneqq l_{m, n} - \sum_{i \in I_{C}} 2^{k_i} \overline{\omega}_{i, m}
    \leq \sum_{i \in I_{\max}} 2^{k} \omega_{i, m}
    \leq u_{m, n} - \sum_{i \in I_{C}} 2^{k_i} \overline{\omega}_{i, m} \eqqcolon u(\overline{\omega}).
\end{equation*}
Clearly, if $[l(\overline{\omega}), u(\overline{\omega})]$ does not contain any lattice points,
the fixing $\overline{\omega}_{I_{C}, m}$ is infeasible and the inequality
\begin{equation*}
    \|{\omega}_{I_{C}, m} - \overline{\omega}_{I_{C}, m} \|_1 \geq 1
\end{equation*}
is valid for~\eqref{eq:knapsack_inequalities}. In order for the interval to
not contain any lattice points, it must hold that
\begin{equation*}
    u_{m, n} - 2^{k_{\max}} + 1 
    \leq \sum_{i \in I_{C}} 2^{k_i} \overline{\omega}_{i, m} + 2^{k_{\max}} y
    \leq l_{m, n} - 1,
\end{equation*}
where $y \in \mathbb{Z}$. We can therefore separate a fractional solution $\omega^{*}$ by
optimizing $\|{\omega}^{*}_{I_{C}, m} - \overline{\omega}_{I_{C}, m} \|_1$ over this set,
solving an integer program with $|I_{C}|$ 0/1 variables, the general integer variable $y$ and 
one linear constraint.

\paragraph{Parity inequalities}

Consider a fixed $m \in [M]$ and an interval $I \coloneqq \{r, \ldots, s\}$ with
$1 \leq r < s \leq n$. By combining inequalities~\eqref{eq:knapsack_inequalities}, we can 
derive that
\begin{equation*}
    l_{s, n} - u_{r, n}
    \leq \sum_{i = r + 1}^{s} 2^{k_i} \omega_{i, n}
    \leq u_{s, n} - l_{r, n}.
\end{equation*}
Based on the interval $I$, we let $k_{\min} \coloneqq \min_{i=r+1}^{s} k_i$ and
$k_{\max} \coloneqq \max_{i=r+1}^{s} k_i$. For any $k \in \{k_{\min}, \ldots, k_{\max}\}$ we
can divide these inequalities by $2^{k}$ and round them to obtain the
valid inequalities
\begin{equation*}
\begin{aligned}
    \sum_{i \in I : k_i \geq k}^{n} 2^{k_i - k} \omega_{i, n} 
    &\leq \lfloor 2^{-k} (u_{s, n} - l_{r, n}) \rfloor \qquad \text{and} \\
    \lceil 2^{-k} (l_{s, n} - u_{r, n}) \rceil
    &\leq \sum_{{i \in I : k_i < k}} \omega_{i, n}
    + \sum_{{i \in I : k_i \geq k}} 2^{k_i - k}\omega_{i, n},
\end{aligned}
\end{equation*}
which can be separated simply by enumerating all intervals in quadratic time.

\subsection{Primal Heuristic}
\label{sec:primal_heuristic}

If $\Omega \subset \R$, switching cost aware rounding can be treated as a shortest path search
on a directed acyclic graph, see \cite{bestehorn2021mixed}. This is because the cost function
of the combinatorial optimization problem is \emph{sequence-dependent}
in the sense of
\cite[Definition 2.47]{bestehorn2021combinatorial}. Specifically, the optimal decision for the
variables corresponding to one grid cell only depend on the optimal choices of previously
decided variables when the grid cells are ordered along the axis $\R$. 
This property does not carry over to the case of rounding grids that decompose multi-dimensional
domains because when ordering the grid cells in a certain way, there will always be parts of the
cost function that depend on the values of the optimization variables that are not close in terms
of the ordering of the grid cells but adjacent in the domain as a subset of $\R^d$, $d \ge 2$.
While the analysis is carried out for uniform rounding grids in \cite{bestehorn2021combinatorial},
the algorithmic approach is still valid for grids, where the volumes of the grid cells are integer
multiples of a positive real scalar, see \cite[Sections 4.4.1 and 6.4.4]{bestehorn2021combinatorial}.

However, one can fix an ordering of the grid cells and only take the parts of the objective
that depend on previously decided variables into account into a decision. In this way,
one obtains a relaxation that allows to compute suboptimal points using the algorithm from
the one-dimensional case for the two-dimensional problem. This has been proposed in 
\cite{bestehorn2021switching} and constitutes a primal heuristic for the switching cost aware
rounding problem.

Moreover, one can improve over this heuristic by basing the decision on all past decisions and the
potential decisions of the next $k$ grid cells along the ordering of the grid cells, thereby neglecting
fewer terms of the objective. In this way one can balance the runtime for the primal heuristic
with the quality of its solution. This decision window of length $k$ is called prefix in 
\cite{bestehorn2021switching}. The heuristic can be improved further: grid cells are removed from
the prefix if all adjacent grid cells were visited, thereby reducing the computational effort.

Because of the curse of dimensionality, the positive effect of the heuristic decreases
with increasing dimension of $\Omega$.

\section{Computational Experiments}\label{sec:computational_experiments}

We consider the topology optimization problem of designing cloaks for 
wave functions governed by the Helmholtz equation in a 2D scenario similar to 
\cite{haslinger2015topology,leyffer2021convergence}.

\subsection{Problem Instances}

For an incident
wave $y_0$ and a design area $D_s \subset \Omega \subset \mathbb{R}^2$, 
we seek a function $v : D_s \to \{\nu_1,\nu_2,\nu_3\}$, where
$\nu_1 = 0$, $\nu_2 = 0.5$, and $\nu_3 = 1$ are possible material
constants with $v(x) = \nu_1$ indicating that no material is placed at $x$.
The goal is to protect an object in another region
$D_o \subset \Omega \subset \mathbb{R}^2$
from the incident wave. The domain together with the areas $D_s$ and $D_o$ 
are depicted in \Cref{fig:scenario}.
Using partial outer convexification, we obtain
the reformulation $v(x) = \sum_{i=1}^M w_i(x) \nu_i$ for a.a.\
$x \in \Omega$, and the considered optimization problem is
\begin{equation}
\label{eq:p_Helmholtz}
\tag{P$_H$}
\begin{aligned}
	\inf_{u,w} \quad  & \frac{1}{2}\|y + y_0\|^2_{L^2(D_o)} \\
	\text{ s.t.\ } \quad & -\Delta y - k_0^2 y = \left(k_0^2 y q + k_0^2 q y_0\right) \sum_{i=1}^M v_i w_i \quad\text{ in } \Omega,
	               \nonumber\\
                 & (\partial y/\partial n) - i k_0 y = 0
                   \quad\text{ on } \partial \Omega, \nonumber\\
	             & w(x) \in \{0,1\}^3 \text{ and } \sum_{i=1}^3w_i(x) = 1
	                \text{ for a.a.\ } x \in D_s
	               \nonumber, \\
	             & w(x) = 0 \text{ for a.a.\ } x \in D \backslash D_s, \nonumber    
\end{aligned}
\end{equation}
where $k_0$ is the wave number ($k_0 = 6\pi$ in our experiments) and
the incident wave is given by $y_0 = \exp(i k_0 d^Tx)$
for $x \in \Omega$ for some direction $d$ (see below) on the unit sphere with $\|d\| = 1$.
We solve the problem based on different choices of the incident wave $y_0$, specifically
for directions $d = (\sin(\delta), \cos(\delta))^T$
for $\delta \in \{0, 5, 10, 15\}$ degrees.
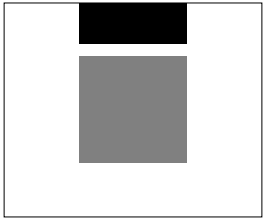
\begin{figure}[hb]
    \begin{center}
    	\begin{tikzpicture}[scale=0.5]
    	\begin{axis}[xtick=\empty,ytick=\empty,xmin=0.0,xmax=1.0,ymin=0.0,ymax=1.0,
    	             axis equal,axis background/.style={fill=white}]]
    		\draw[color=white!50!black,fill=white!50!black,line width=0pt,draw=none] (axis cs:.25,.25) rectangle (axis cs:.75, .75);
    		\draw[color=black,fill=black,line width=0pt,draw=none]
    			(axis cs:.25,.8125) rectangle (axis cs:.75,1.);
    		\end{axis}		
    	\end{tikzpicture}
    	\caption{Domain $\Omega = (0,2)^2$ with the scatterer design area
    	$D_s = (.5,1.5)^2$ in light gray and the protected area
    	$D_o = (0.5,1.5)\times (1.625,2.)$ in black.}
    \label{fig:scenario}
    \end{center}
\end{figure}

\subsection{Discretization and Solution of the Relaxed Instances}

We follow the setup of our example from \cite{bestehorn2021switching}.
We choose a uniform discretization of $\Omega$ a grid of $2^k\times 2^k$ squares.
Each of the squares is split into $4$ triangles. Starting from this grid,
we solve the discretized Helmholtz equation with the open-source library \textsc{Firedrake}
\cite{Rathgeber2016}, where we use \textsc{PETSc} numerical linear
algebra backend \cite{petsc-user-ref}. The adjoint equation for the
discretized reduced objective is computed using \textsc{dolfin-adjoint} 
\cite{farrell2013automated}.

Due to the choice of design area, a fourth of the rectangles in
the discretization, i.e., $N(k) = 2^{2k - 2}$ in total, are available
to place material. We solve the problem based on the values $k = 6, 7$,
resulting in $N(k) = 1024, 4096$ cells. Together with the choices of $\delta$,
this yields a total of eight
computational instances, which are then solved to local optimality.

\subsection{Computing Binary from Relaxed Controls}
After solving the continuous relaxations instances to local optimality,
we compute non-uniform grids by means of~\Cref{alg:adaptive_grid_refinement}
with $\RA = \Cref{alg:sur} $.
Then we solve the corresponding switching cost aware rounding problems
on the uniform and the non-uniform grids. To this end, we model the objective
\eqref{eq:scarp_objective} by means of additional variables and linear inequalities.
Specifically, we introduce variables $\lambda_{i,j}^{m} \geq 0$ 
for all $1 \leq i < j < N$ and $m \in [M]$ where
$\Ha^{d-1}(\overline{T_i} \cap \overline{T_j}) \neq 0$
and require that
\begin{equation}
    \label{eq:ip_cost_cons}
    \lambda_{i,j}^{m} \geq \omega_{i, m} - \omega_{j, m} 
    \enskip\text{and}\enskip
    \lambda_{i,j}^{m} \geq \omega_{j, m} - \omega_{i, m}.
\end{equation}
For our experiments, we set the cost coefficients to $c \equiv 1$.
Together with the constraints above, this yields a 
mixed-integer program with $NM$ binary control variables $\omega$ constrained
by~\eqref{eq:ip_cons} as well as a number of cost variables $\lambda$ coupled
to the control variables based on constraints~\eqref{eq:ip_cost_cons}. 

We solve these instances 
using \textsc{SCIP}~8.0.3~\cite{SCIP} and \textsc{Gurobi}~9.1~\cite{gurobi} as underlying LP solver setting a time limit of \SI{3600}{\second}
for each instance. As a baseline, we used as a rounding grid the completely refined grid consisting of $N(k)$ cells,
yielding the results in~\Cref{table:baseline_uniform}. We see that most instances
are not solved to optimality within the given time limit. The gaps between
the best primal solution objective $p$ and the best dual objective $d$,
defined as $(p - d) / d$ are well above \SI{100}{\percent} for the \emph{small}
instances with $k = 6$ and no lower bounds were computed within the prescribed
time limit for the \emph{large} ones with $k = 7$.
This is likely due to the problems size, which, in terms of the number of problem variables,
is $\approx\num{20000}$ for the large instances.
\begin{table}[h]
    \centering
    \begin{tabular}{rrS[table-format=2.4,round-mode=places,round-precision=4]S[table-format=2.4,round-mode=places,round-precision=4]S[table-format=2.4,round-mode=places,round-precision=2]}
    \toprule
        \multicolumn{2}{c}{{\textbf{Instance}}} & {\textbf{Primal bound}} & {\textbf{Dual bound}} & {\textbf{Gap}}  \\
        k & $\delta$ \\
    \midrule
    \multirow{4}{*}{6}
        &  0 & 21.1875 & 9.16863 & 1.31087 \\
        &  5 & 29.1875 & 10.7325 & 1.71955 \\
        & 10 & 28.75 & 10.6772 & 1.69266 \\
        & 15 & 28.5 & 9.84206 & 1.89573 \\
    \midrule
    \multirow{4}{*}{7}
        &  0 & 37.1875 & {--} & {--} \\
        &  5 & 45.75 & {--} & {--} \\
        & 10 & 52.4062 & {--} & {--} \\
        & 15 & 49.7188 & {--} & {--} \\
    \bottomrule
    \end{tabular}
    \caption{Baseline solutions on uniform grids.}
    \label{table:baseline_uniform}
\end{table}
The results for the non-uniform grids computed by~\Cref{alg:adaptive_grid_refinement}
are shown in~\Cref{table:baseline}. The number of cells in the final refinement is about half
compared to the original number $N(k)$, yielding a similar reduction in the number of variables in
the corresponding program. As a result, the
problems exhibit generally lower runtimes, leading to nontrivial
dual bounds even for the large instances. On the other hand, the less structured programs appear to 
impede the built-in primal heuristics, with no primal bounds being found for the large
instances. This situation is easily remedied by using the solution provided by the
initial heuristic described in \Cref{sec:primal_heuristic}, leading to the results
in \Cref{table:primal}, with both primal and dual bounds being found for each
instance within the prescribed time limit. Unfortunately, however, the resulting gaps remain
substantial, ranging from  \SI{38}{\percent} for a \emph{small} 
instance to about \SI{162}{\percent} for a \emph{large} one. The inclusion of the valid
inequalities introduced in \Cref{sec:valid_ineqaulities} yields a significant improvement
in this respect (see \Cref{table:full}).
\begin{table}[h]
    \centering
    \begin{tabular}{rrS[table-format=4]S[table-format=2.4]S[table-format=2.4]S[table-format=2.4,round-mode=places,round-precision=2]}
    \toprule
        \multicolumn{2}{c}{{\textbf{Instance}}} & {\textbf{N}} & {\textbf{Primal bound}} & {\textbf{Dual bound}} & {\textbf{Gap}}  \\
        k & $\delta$ \\        
    \midrule
    \multirow{4}{*}{6}
        &  0 & 652 & 21.5625 &12.447 & 0.732348 \\
        &  5 & 607 & 18.5625& 15.1035 & 0.229018 \\
        & 10 & 640 & 18.25 & 13.7847 & 0.32393 \\
        & 15 & 676 & 22.625 & 11.4893 & 0.969228 \\
    \midrule
    \multirow{4}{*}{7}
        &  0 & 1666 & {--} & 16.321 & {--} \\
        &  5 & 1825 & {--} & 16.65 & {--} \\
        & 10 & 1936 & {--} & 15.0073 & {--} \\
        & 15 & 1831 & {--} & 15.5981 & {--} \\
    \bottomrule
    \end{tabular}
    \caption{Baseline solutions on refined grids.}
    \label{table:baseline}
\end{table}
\begin{table}[h]
    \centering
    \begin{tabular}{rrS[table-format=4]S[table-format=2.4]S[table-format=2.4]S[table-format=2.4,round-mode=places,round-precision=2]}
    \toprule
        \multicolumn{2}{c}{{\textbf{Instance}}} & {\textbf{N}} & {\textbf{Primal bound}} & {\textbf{Dual bound}} & {\textbf{Gap}}  \\
        k & $\delta$ \\        
    \midrule
    \multirow{4}{*}{6}
        &  0 & 652 & 19.3125 & 12.4472 & 0.551555 \\
        &  5 & 607 & 23.125 & 13.3345 & 0.734224 \\
        & 10 & 640 & 18.5625 & 13.4894 & 0.376085 \\
        & 15 & 676 & 20.125 & 11.8544 & 0.697675 \\
    \midrule
    \multirow{4}{*}{7}
        &  0 & 1666 & 30.1562 & 16.458 & 0.83232 \\
        &  5 & 1825 & 35.3438 & 16.5781 & 1.13196 \\
        & 10 & 1936 & 39.5 & 15.0614 & 1.6226 \\
        & 15 & 1831 & 37.6875 & 15.6568 & 1.4071 \\
    \bottomrule
    \end{tabular}
    \caption{Solutions with primal heuristic on refined grids.}
    \label{table:primal}
\end{table}
\begin{table}[h]
    \centering
    \begin{tabular}{rrS[table-format=4]S[table-format=2.4]S[table-format=2.4]S[table-format=2.4,round-mode=places,round-precision=2]}
    \toprule
        \multicolumn{2}{c}{{\textbf{Instance}}} & {\textbf{N}} & {\textbf{Primal bound}} & {\textbf{Dual bound}} & {\textbf{Gap}}  \\
        k & $\delta$ \\        
    \midrule
    \multirow{4}{*}{6}
        &  0 & 652 & 18.6875 & 14.8329 & 0.259869 \\
        &  5 & 607 & 18.375 & 15.7063 & 0.16991 \\
        & 10 & 640 & 17.75 & 14.594 & 0.216255 \\
        & 15 & 676 & 19.0625 & 13.8066 & 0.380677 \\
    \midrule
    \multirow{4}{*}{7}
        &  0 & 1666 & 30.1562 & 18.5489 & 0.625771 \\
        &  5 & 1825 & 35.3438 & 18.4887 & 0.91164 \\
        & 10 & 1936 & 39.5 & 17.0939 & 1.31076 \\
        & 15 & 1831 & 37.6875 & 17.6582 & 1.13428 \\
    \bottomrule
    \end{tabular}
    \caption{Solutions with primal heuristic and separators on refined grids.}
    \label{table:full}
\end{table}

In line these numerical improvements, there are non-negligible visible differences computed controls.
In order to provide an impression, we plot the computed controls
and the absolute values of the real part of the corresponding PDE solutions for the case
$k = 7$ and $\delta = 0$ in \Cref{fig:controls_and_states}.
\begin{figure}[h]
    \centering
    \begin{subfigure}{.32\linewidth}
        \centering
        \includegraphics[width=\textwidth]{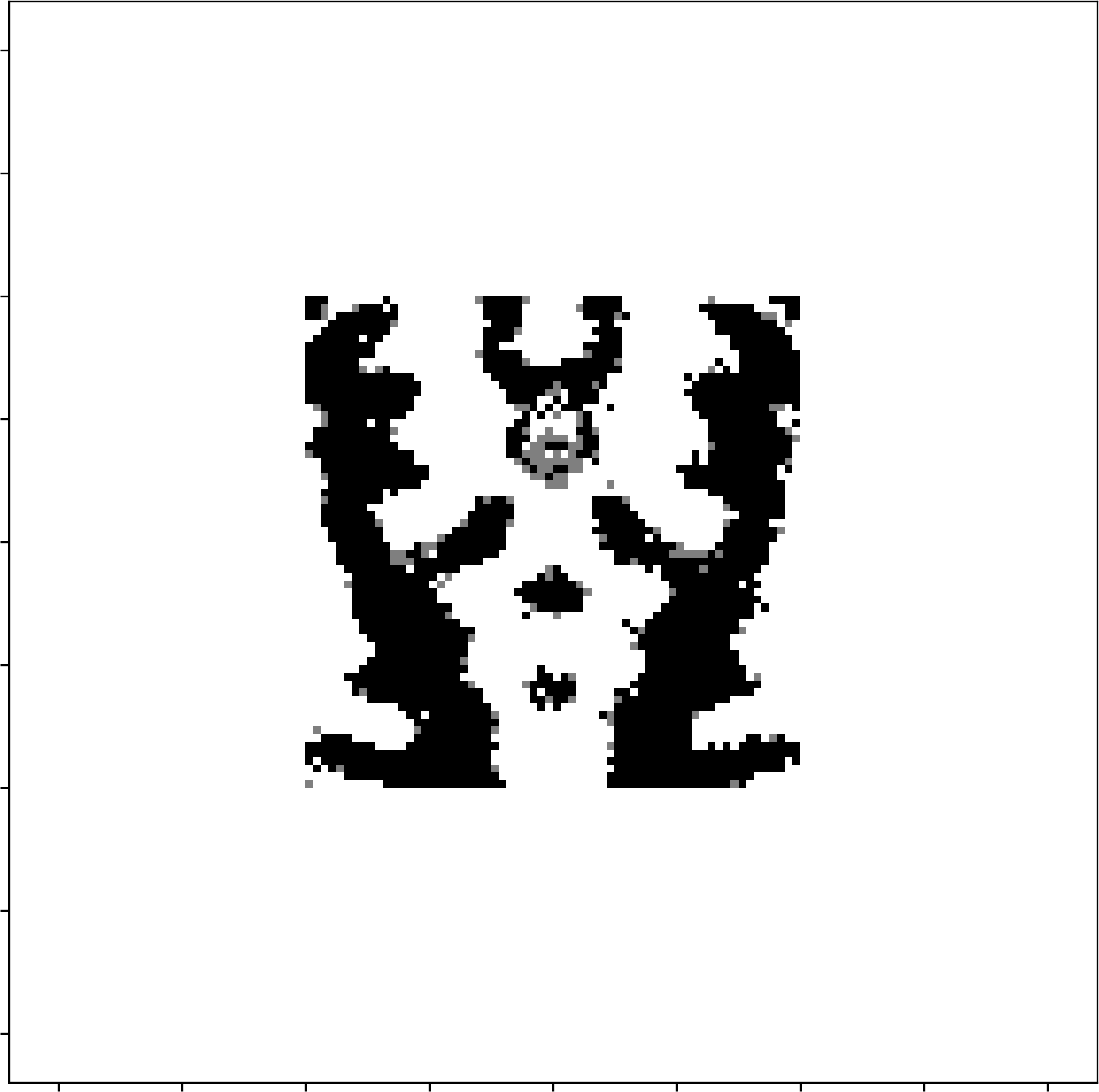}
    \end{subfigure} %
    \hfill
    \begin{subfigure}{.32\linewidth}
        \centering
        \includegraphics[width=\textwidth]{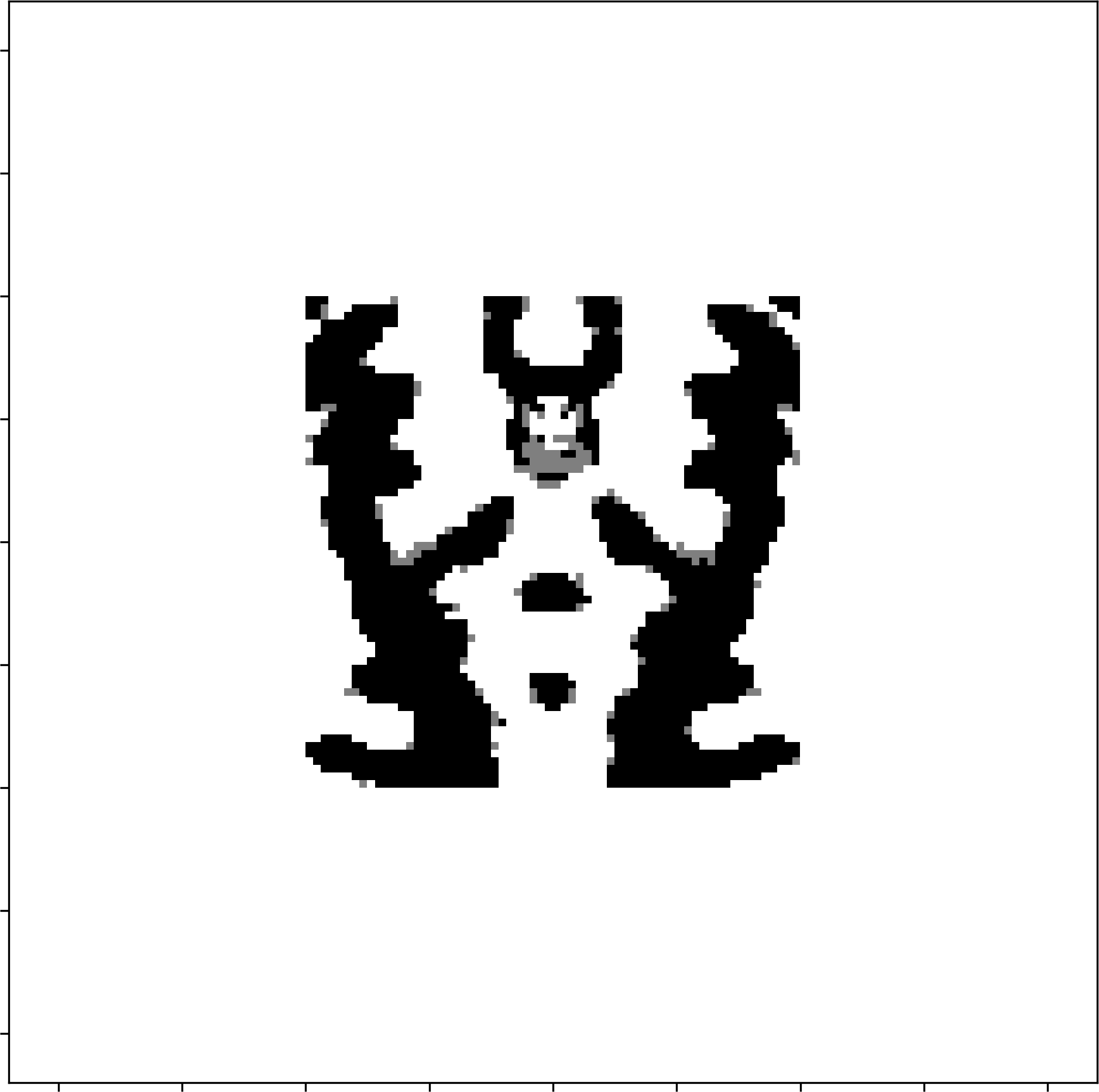}
    \end{subfigure}
    \hfill
    \begin{subfigure}{.32\linewidth}
        \centering
        \includegraphics[width=\textwidth]{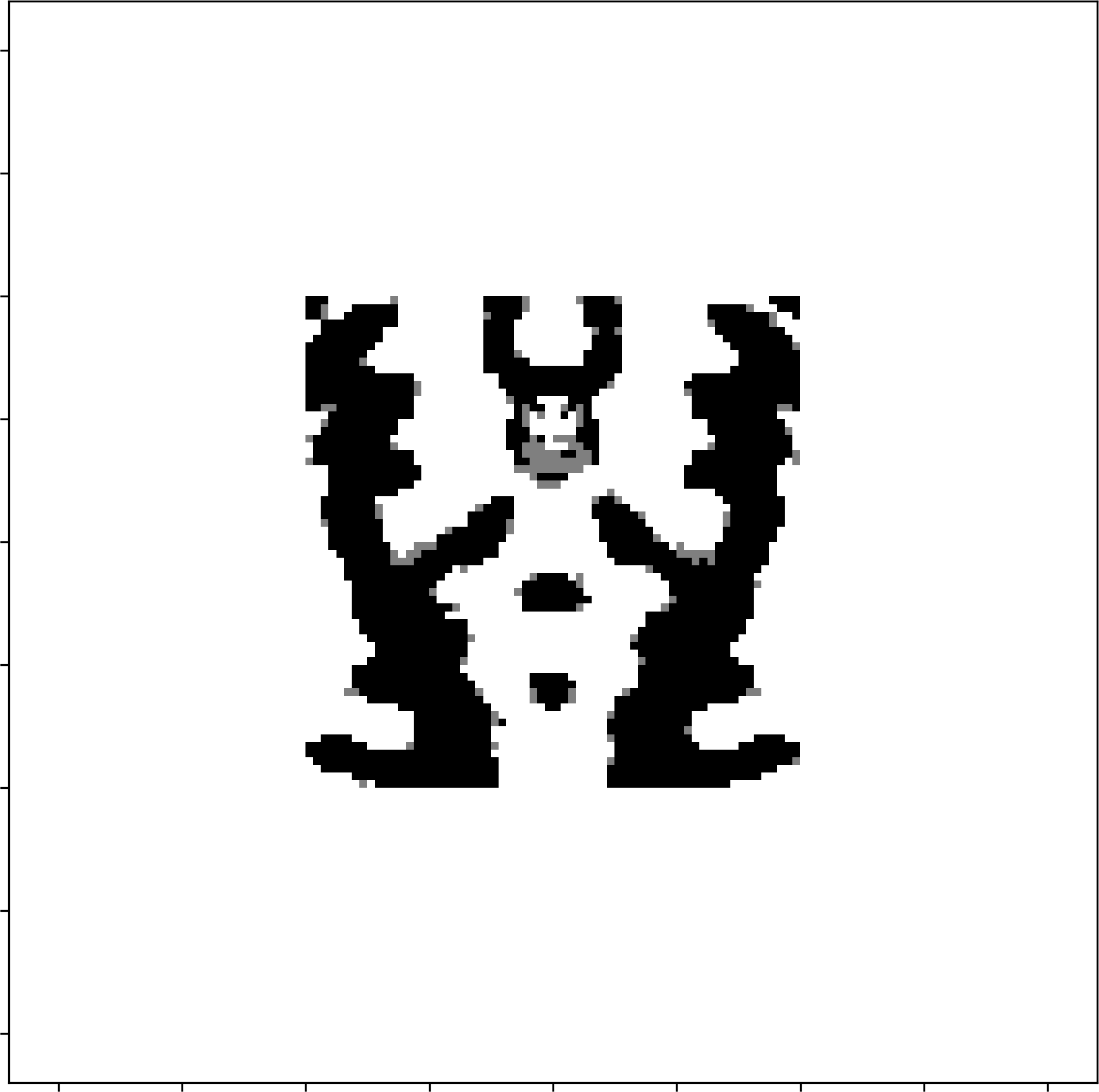}
    \end{subfigure}
    
    \begin{subfigure}{.32\linewidth}
        \centering
        \includegraphics[width=\textwidth]{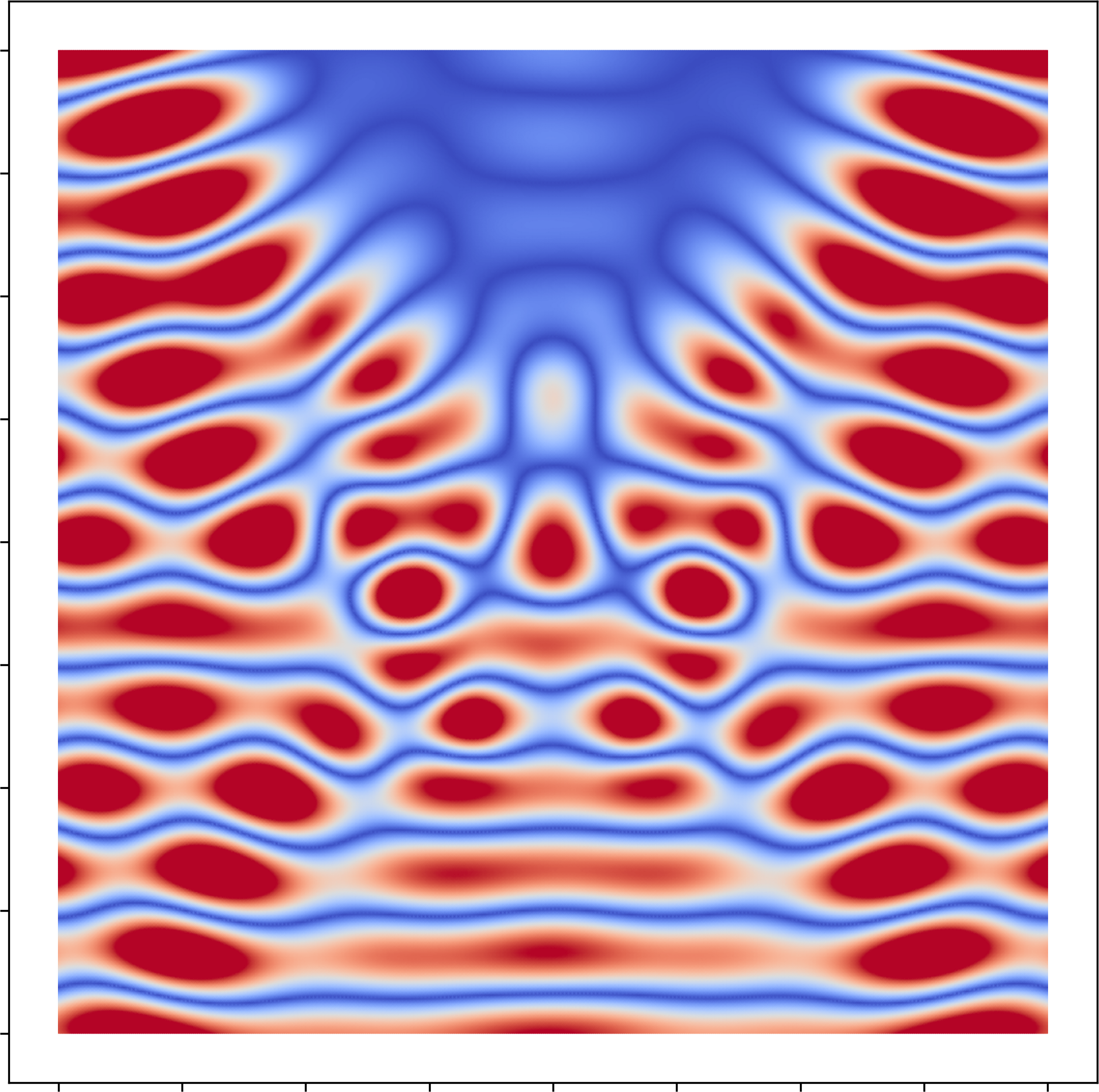}
        \caption{}\label{fig:img_baseline_uniform}
    \end{subfigure} %
    \hfill
    \begin{subfigure}{.32\linewidth}
        \centering
        \includegraphics[width=\textwidth]{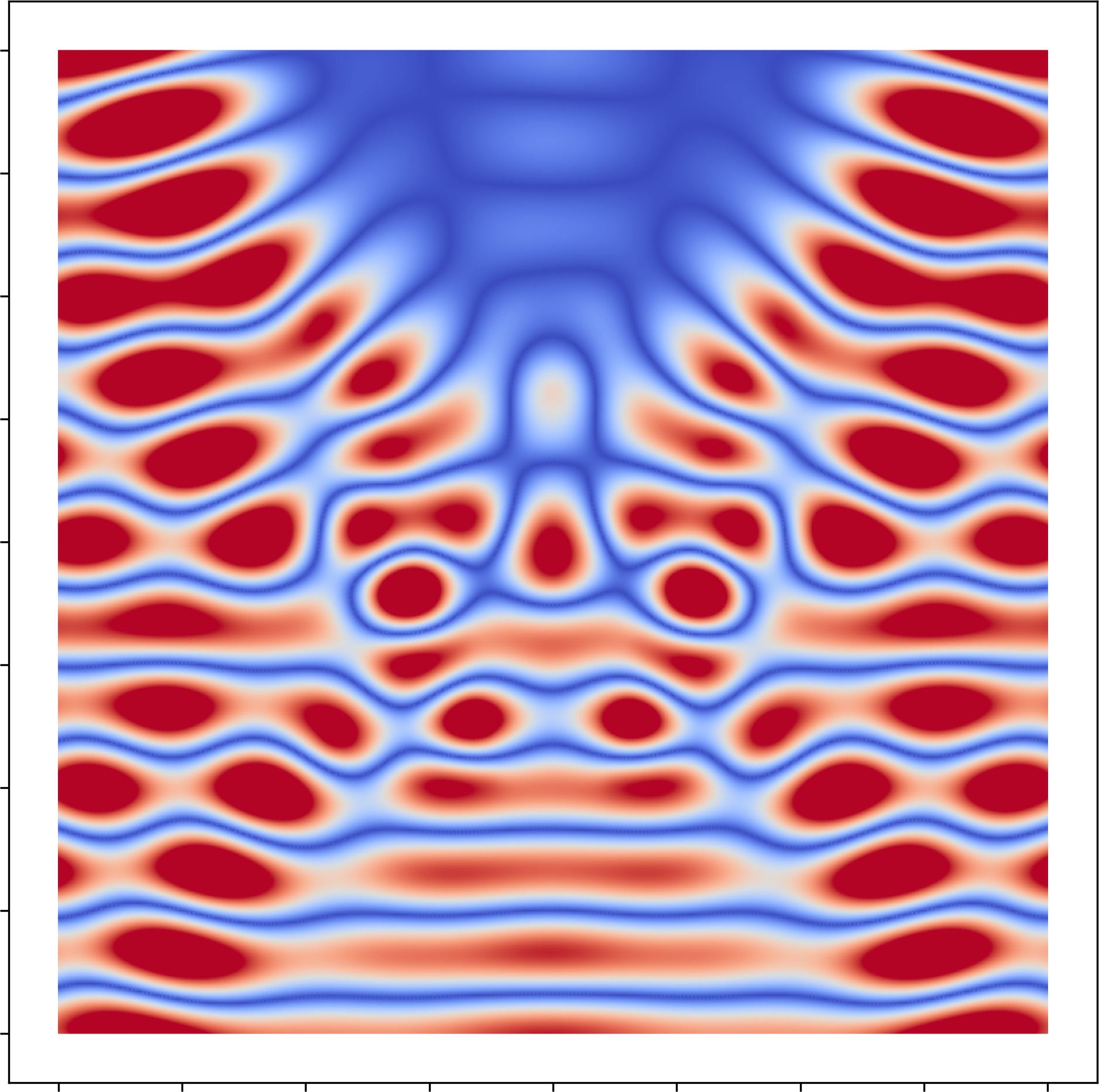}
        \caption{}\label{fig:primal_heuristic}
    \end{subfigure}
    \hfill
    \begin{subfigure}{.32\linewidth}
        \centering
        \includegraphics[width=\textwidth]{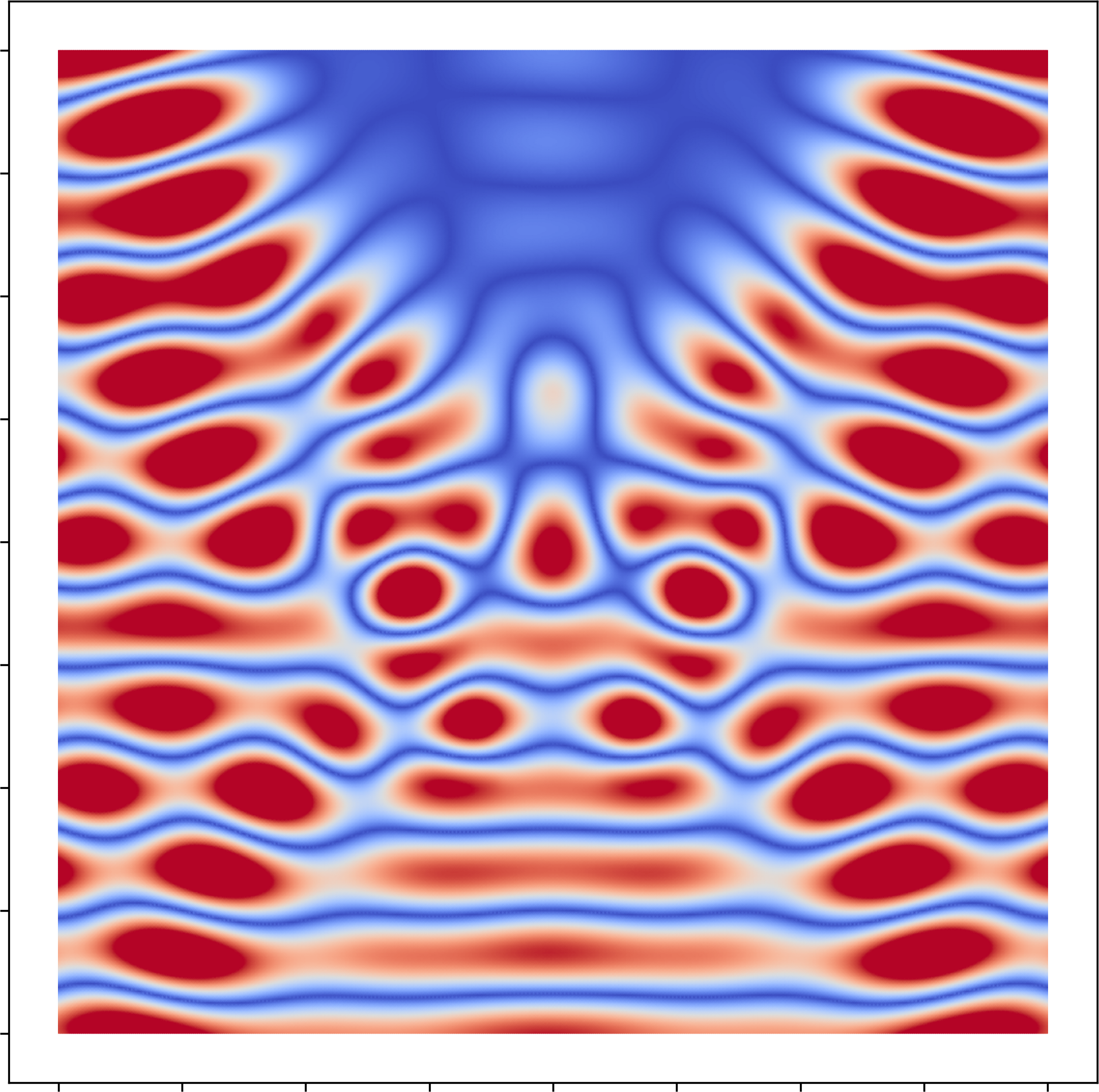}
        \caption{}\label{fig:full}
    \end{subfigure}

    \RawCaption{\caption{
    Computed controls and corresponding moduli of
    the real parts of the solution of the PDE
    constraining \eqref{eq:p_Helmholtz} for $\delta = 0$ and $k = 7$.
    \eqref{fig:img_baseline_uniform}:
    Baseline solution on uniform grid,
    \eqref{fig:primal_heuristic}:
    Solution with primal heuristic on non-uniform grid,
    \eqref{fig:full}:
    Solution with primal heuristic and separators
    on non-uniform grid.
    Note that no baseline solution could be computed on the non-uniform grid
	 within the prescribed time limit of one hour.}
    \label{fig:controls_and_states}}
\end{figure}

\section{Acknowledgment}
C.~Kirches, F.~Bestehorn and P.~Manns acknowledge funding by Deutsche For\-schungsgemeinschaft (DFG) through Priority Programme 1962 (grant Ki 1839/1-2).

\section{Conclusion}
We proposed a non-uniform grid refinement strategy that preserves the approximation properties
of the combinatorial integral approximation. Using these grids inside switching cost aware
rounding on instances of a test problem that is defined on a multi-dimensional domain, we obtained
mixed-integer programs with fewer variables and better tractability, i.e.\ better runtime performance,
when treating them with a general purpose integer programming solver. We were generally
able to improve the solution process of the integer program by adding valid linear inequalities
and applying the efficient combinatorial algorithm that is available for the one-dimensional case as a heuristic
to the multi-dimensional problem.
However, the considered instances can still not be solved to global optimality within acceptable
time limits and large duality gaps remain. Moreover, the visible difference between the
computed solutions are very small.

\bibliographystyle{plain}
\bibliography{references}

\end{document}